\documentclass[preprint,12pt]{elsarticle}
\usepackage{tikz}
\usepackage{amsthm,amsfonts,amssymb,amscd,amsmath,enumerate,verbatim,calc,graphicx,geometry}
\usepackage[all]{xy}
\newtheorem{theorem}{Theorem}[section]

\theoremstyle{definition}
\theoremstyle{definitions}
\newtheorem{definition}[theorem]{Definition}

\newtheorem{remark}[theorem]{Remark}

\theoremstyle{notations}

\theoremstyle{remarks}

\journal{}

\begin{document}

\begin{frontmatter}



\title{Covering Maps with respect to Topologies on the Fundamental Group }


\author[]{Naghme Shahami }
\ead{na.shahami@mail.um.ac.ir }
\author[]{Behrooz Mashayekhy\corref{cor1}}
\ead{bmashf@um.ac.ir }

\address{Department of Pure Mathematics, Center of Excellence in Analysis on Algebraic Structures, Ferdowsi University of Mashhad,\\
P.O.Box 1159-91775, Mashhad, Iran.}
\cortext[cor1]{Corresponding author}

\begin{abstract}
In this paper, using the classical covering theory, we introduce a generalization of covering maps of a space $X$ with respect to a topology $\tau$ 
on the fundamental 
group of $X$. We show that the famous notions, covering, semicovering, generalized covering and fibration maps are special cases of this new 
notion $\pi_1^{\tau}$-covering map.  Moreover, among presenting some properties for this new notion, 
we compare $\pi_1^{\tau}$-covering maps of a space $X$ for 
several famous topologies on the fundamental group of $X$.
\end{abstract}

\begin{keyword}
Covering map\sep  semicovering\sep  generalized covering\sep  fundamental group\sep  compact-open topology\sep  
Spanier topology\sep  whisker topology\sep  lasso topology\sep  subgroup topology.

\MSC[2020]{57M10, 57M12, 57M05, 55Q05.}
\end{keyword}

\end{frontmatter}




\section{Introduction and Motivation}

We recall that a continuous map $p:\widetilde{X}\rightarrow X$ is a covering map if $\widetilde{X}$ is a path connected and every point of $X$ 
has an open neighborhood 
$U$ which is evenly covered by $p$, that is, $p^{-1}(U)$ is a disjoint union of open subsets of $\widetilde{X}$ such that each of which 
is mapped homeomorphically onto 
$U$ by $p$. In classical covering theory, the existence of a covering map $p:(\widetilde{X},\widetilde{x}_0)\rightarrow (X,x_0)$ for which the equality 
$G=\pi_1(\widetilde{X},\widetilde{x}_0)$ holds, where $G$ is a subgroup of $\pi_1(X,x_0)$,  is an essential question. 
Spanier \cite[p. 82]{Span} gave an answer 
to the above question as follows (see also Rotman \cite[p. 295]{Rot}). 

Let $G$ be a subgroup of $\pi_1(X,x_0)$ and let $P(X,x_0)$ be the family of all paths $f$ in $X$ with $f(0)=x_0$. 
Define $f\sim_G g$ by $f(1)=g(1)$ and 
$[f\ast g^{-1}]\in G$. It is easy to that $\sim_G$ is an equivalent relation on $P(X,x_0)$. Denote the equivalence class of $f\in P(X, x_0)$ by 
$\langle f\rangle_G$ 
and define $\widetilde{X}_G$ as the set of all such equivalence classes. Define $\widetilde{x}_0=\langle c_{x_0}\rangle_G$, 
where $c_{x_0}$ is the constant path at $x_0$, 
and a function $p_G:(\widetilde{X}_G,\widetilde{x}_0)\rightarrow (X,x_0)$ by $\langle f\rangle_G\mapsto f(1)$, 
which is called the endpoint projection. 
By considering the whisker topology on $\widetilde{X}_G$ 
which is generated by the sets $(U,\langle f\rangle_G)=\{ \langle f\ast\lambda \rangle_G\mid \lambda([0,1])\subseteq U\}$ as a basis, the function 
$p_G:(\widetilde{X}_G,\widetilde{x}_0)\rightarrow (X,x_0)$ is continuous and also is onto if $X$ is path connected 
(see \cite[Lemma 10.31]{Rot}). Also, 
every path $f$ in $X$ beginning at $x_0$ can be lifted to a path $\widetilde{f}$ in $\widetilde{X}_G$ beginning at $\widetilde{x}_0$ and ending 
at $\langle f\rangle_G$ 
defined by $\widetilde{f}(t)=\langle f_t\rangle_G$, where $f_t(s)=f(ts)$ (see \cite[Lemma 10.32]{Rot}). Note that $\widetilde{f}$ is called 
the standard lift of $f$ and 
hence the map $p_G:(\widetilde{X}_G,\widetilde{x}_0)\rightarrow (X,x_0)$ has the path lifting property. As a consequence, one can see easily that 
$\widetilde{X}_G$ is path connected. 
As a classical result, it is proved that if $X$ is connected, locally path connected, and semilocally simply connected, 
then $p_G:(\widetilde{X}_G,\widetilde{x}_0)\rightarrow (X,x_0)$ is a covering map and $\pi_1(p_G)(\pi_1(\widetilde{X}_G,\widetilde{x}_0))=G$ 
(see \cite[Theorem 10.34]{Rot}). 
More generally, it is proved that if $X$ is connected, locally path connected, and if there is an open covering $\mathcal{U}$ of $X$ such that 
$\pi(\mathcal{U},x_0)\subseteq G$, then $p_G:(\widetilde{X}_G,\widetilde{x}_0)\rightarrow (X,x_0)$ is a covering map and 
$\pi_1(p_G)(\pi_1(\widetilde{X}_G,\widetilde{x}_0))=G$, 
where $\pi(\mathcal{U},x_0)$, the Spanier subgroup of the open covering $\mathcal{U}$, is the normal subgroup of $\pi_1(X, x_0)$ generated 
by the homotopy class 
of lollipops $\alpha * \beta * \alpha^{-1}$, where $\beta$ is a loop lying in an element of $U \in \mathcal{U}$ at $\alpha(1)$, and $\alpha$ is 
any path originated 
at $x_0$ (see \cite[Section 2.5 Theorem 13]{Span}). Conversely, if $p:(\widetilde{X},\widetilde{x}_0)\rightarrow (X,x_0)$ is a covering map with 
$G=\pi_1(p)(\pi_1(\widetilde{X},\widetilde{x}_0))$, 
then there exists an open covering $\mathcal{U}$ of $X$ such that $\pi(\mathcal{U},x_0)\subseteq G$ (see \cite[Section 2.5 Theorem 12]{Span}). 
Note that the \emph{Spanier topology} \cite[p. 12]{W} on the fundamental group $\pi_1(X, x_0)$, denoted by $\pi_1^{Span}(X, x_0)$, 
was introduced using the collection of all Spanier subgroups $\pi(\mathcal{U}, x_0)$. By considering the Spanier topology  $\pi_1^{Span}(X, x_0)$ 
and the above two last classical results, for a connected and locally path connected space $X$, there exists a covering map 
$p:(\widetilde{X},\widetilde{x}_0)\rightarrow (X,x_0)$ 
with $G=\pi_1(p)(\pi_1(\widetilde{X},\widetilde{x}_0))$ if and only if $G$ is an open subgroup of $\pi_1^{Span}(X, x_0)$.

The above properties motivate us to introduce a generalization of covering maps of a space $X$ with respect to a topology $\tau$ on the 
fundamental group 
$\pi_1(X,x_0)$ as follows.

\begin{definition}
Let $\tau$ be a topology on the fundamental group $\pi_1(X,x_0)$, denoted by $\pi_1^{\tau}(X,x_0)$. Then we call a map 
$p: (\widetilde{X},\widetilde{x}_0) \to (X,x_0)$ 
 a ``$\pi_1^{\tau}$-covering map" if the following hold:
\begin{itemize}
\item $p$ is continuous.
\item $\widetilde{X}$ is path connected.
\item $p$ has path lifting property.
\item $\pi_1(p)(\pi_1(\widetilde{X},\widetilde{x}_0))$ is open in $\pi_1^{\tau}(X,x_0)$.
\end{itemize}
\end{definition}

\begin{remark}\label{1.2}
Note that using path lifting property, it is easy to see that a $\pi_1^{\tau}$-covering map is onto. Moreover, 
for any open subgroup $G$ of $\pi_1^{\tau}(X,x_0)$, we can show that the endpoint projection 
$p_G:(\widetilde{X}_G,\widetilde{x}_0)\rightarrow (X,x_0)$ 
is a $\pi_1^{\tau}$-covering map if $\widetilde{X}$ is path connected and $\pi_1^{\tau}(X,x_0)$ is a semitopological group. 
By the above argument $p_G$ has the path lifting property. If $[f]\in G$, then $f\sim_G c_{x_0}$ and so 
$\langle f\rangle_G=\langle c_{x_0}\rangle_G=\widetilde{x}_0$. Thus $\widetilde{f}(1)=\widetilde{f}(0)$ and hence 
$\pi_1(p_G)([\widetilde{f}])=[f]$, where 
$\widetilde{f}$ is the standard lift of $f$. Therefore we have $G\subseteq \pi_1(p_G)(\pi_1(\widetilde{X}_G,\widetilde{x}_0))$.  
Hence $p_G$ is a $\pi_1^{\tau}$-covering map. Also, note that Brazas in \cite[Lemma 5.9]{B15} proved that  
$p_G:(\widetilde{X}_G,\widetilde{x}_0) \to (X,x_0)$ has unique path lifting property if and only if 
$G= \pi_1(p_G)(\pi_1(\widetilde{X}_G,\widetilde{x}_0))$ if and only if $p_G$ has unique lifting property.
\end{remark}

Some people extended the notion of covering maps in various viewpoints, for instances, rigid covering maps \cite{Biss}, semicovering maps \cite{B12},
generalized covering maps \cite{B15,FZ7} and fibration \cite{Span}. 
These generalizations focus on keeping some properties of covering maps and eliminating the evenly covered property. 
Brazas \cite{B12} defined semicovering maps by 
removing the evenly covered property and keeping local homeomorphism and the unique path and homotopy lifting properties. 
For generalized covering maps, 
the local homeomorphism is replaced with the unique lifting property (see \cite{B15,FZ7}). 
We show that the famous notions, covering, semicovering, generalized covering and fibration maps are special cases of this new 
notion $\pi_1^{\tau}$-covering map for some suitable topologies ${\tau}$. 

In order to investigate the notion $\pi_1^{\tau}$-covering maps, we need to review some famous topologies on the 
fundamental group $\pi_1(X,x_0)$, 
which is done in Section 2. In Section 3, in order to dealing with the notion $\pi_1^{\tau}$-covering maps, we study and present some properties of 
$\pi_1^{\tau}$-covering maps, specially, we show that  the new  notion $\pi_1^{\tau}$-covering maps include some well known generalizations 
of classical covering maps 
such as semicovering, generalized covering and fibration maps.
Finally, in Section 4, we compare $\pi_1^{\tau}$-covering maps of a space $X$ for several famous topologies  on the fundamental group 
of $\pi_1(X,x_0)$.


\section{ Some Topologies on Fundamental Groups} \label{Sec2}

 Putting natural topologies on the fundamental group comes back to Hurewicz \cite{Hur} in 1935 and Dugundji \cite{D} in 1950. 
Recently, some researchers have been shown that there are various useful, interesting and, functorial topologies on the fundamental group 
which make it a powerful tool for studying some  topological properties of spaces 
(see \cite{A16,A20,B11,B12,B13,B14,B15,CM,F9,F11,J,MPT,PTM17,PTM20,TPM,VZ14}).  
Some well-known topologies on the fundamental group $\pi_1(X,x_0)$ are as follows (see also \cite{SM}).

1. \textbf{The subgroup topology}: A collection $\Sigma$ of subgroups of $G$ is called a \emph{neighbourhood family} if for any $H,K \in \Sigma$, 
there is a subgroup $S \in \Sigma$ such that $S \subseteq H\cap K$. As a result of this property, the collection of all left cosets of elements 
of $\Sigma$ forms 
a basis for a topology on $G$, which is called \emph{the subgroup topology} determined by $\Sigma$ and we denote it by $G^{\Sigma}$. 
The subgroup topology on a group $G$ specified by a neighbourhood family was defined in \cite[Section 2.5]{BS} and considered by some 
recent researchers 
such as \cite{A20,SM,W}.
Since left translation by elements of $G$ determine self-homeomorphisms of $G$, they are homogeneous spaces. 
 Bogley et al. \cite{BS} focused on some general properties of subgroup topologies and by introducing the intersection 
 $S_\Sigma=\cap \{H \ \mid \ H \in \Sigma\}$, called the \emph{infinitesimal} subgroup for the neighbourhood family $\Sigma$, 
 they showed that the closure of the element $g\in G$ is the coset $g S_\Sigma$. 

Let $H$ be a subgroup of a group $G$. Then we define $\Sigma^H$ as $\Sigma^H= \{K\leqslant G \ \mid \ H\subseteq K\}.$
It is easy to see that $\Sigma^H$ is a neighbourhood family. We consider the subgroup topology on $G$ determined by $\Sigma^H$ and denote 
it by $G^H$. 
Note that the infinitesimal subgroup for the neighbourhood family $\Sigma^H$ is $H$. Let $(X,x_0)$ be a pointed topological space and $H$ 
be a subgroup of
 $\pi _1(X,x_0)$, then we denote $\pi_1^H(X,x_0)$ as subgroup topology on $\pi_1(X,x_0)$ determined by $\Sigma^H$.

2. \textbf{The whisker topology}: Let $(X,x_0)$ be a pointed topological space. Define 
$\widetilde{X} = \{[\alpha ] \mid  \alpha :I \rightarrow X , \alpha (0) = x_0\}$ 
as the set of all path-homotopy classes of paths in $X$ starting at $x_0$. Spanier \cite[p. 82]{Span} introduced a topology on $\widetilde{X}$ 
by basic open neighbourhoods 
of $[\alpha]$ of the form $N([\alpha],U) = \{[\alpha \ast \delta] \mid \delta : I \rightarrow U , \delta (0)=\alpha (1) \}$, where $U$ is an 
open neighbourhood of $\alpha (1)$ 
in $X$. This topology on $\widetilde{X}$ is called by Brodskiy et al. \cite{Br12} \emph{the whisker topology} due to $N([\alpha],U)$ consists 
of only homotopy classes 
that differ from $ [\alpha]$ by a small change or “whisker” at it’s end.
The function $p:\widetilde{X} \to X$ as the endpoint projection $p([\alpha])=\alpha (1)$ is continuous and if $X$ is connected, locally path connected, 
and semilocally simply connected, then $p$ is the universal covering map. Clearly $p^{-1}(\{x_0\})=\pi_1(X,x_0)$ and so $\pi_1(X,x_0)$
 is a subspace of $\widetilde{X}$. 
One can consider the subspace topology on $\pi_1(X,x_0)$  inherited from $\widetilde{X}$ which is called \emph{the whisker topology} 
on the fundamental group denoted 
by $\pi_1^{wh}(X, x_0)$  (see \cite{Br12}).

3. \textbf{The compact-open quotient topology}:  Let $(X,x_0)$ be a pointed topological space and $\Omega (X,x_0)$ denote the space of loops in 
$X$ based at $ x_0$. 
There exists the usual compact-open topology on $\Omega (X,x_0)$ which is generated by subbasis sets 
$\langle K,U\rangle=\{\alpha\mid \alpha(K)\subseteq U\}$ for 
compact $K\subseteq [0,1]$ and open $U\subseteq X$. By considering the surjection map $q:\Omega(X,x_0)\to\pi_1(X,x_0),\ q(\alpha)=[\alpha]$ 
one can equip 
$\pi_1(X,x_0)$  with the quotient topology with respect to the map $q:\Omega(X,x_0)\to\pi_1(X,x_0)$ which is denoted by $\pi_{1}^{qtop}(X,x_0)$. 
We refer to this topology as the \emph{natural quotient topology} on $\pi_{1}(X,x_0)$ (see \cite{Biss, B11, F9}).

4. \textbf{The lasso topology}: For any topological space $X$, Brodskiy et al. \cite[Section 3]{Br8} introduced the lasso topology on
 the universal path space 
$\widetilde{X}$ by the basis $  N(\langle \alpha \rangle , \mathcal{U}, W)$, where $ \alpha $ is a path originated at $ x_0 $,  
$W$ is a neighbourhood of the 
endpoint $\alpha(1)$ and $\mathcal{U}$ is an open cover of $X$. A class $\langle \gamma \rangle \in \widetilde{X}$ belongs  to 
$ N(\langle \alpha \rangle , \mathcal{U}, W)$ if and only if this class has a representative of the form $\alpha * L * \beta$ where $L$ belongs to 
$\pi\big(\mathcal{U},\alpha(1)\big)=\{[\prod_{j=1}^{n}u_jv_ju_j^{-1}]\mid u_j\ \emph{are arbitrary paths starting at}\ \alpha(1)\ 
\emph{and each}\ v_j $ 
$\emph{is a loop inside one of}\  U_j\in \mathcal{U}\}$  and $\beta$ is a based loop in $W$ at $\alpha(1)$. There is a bijection between 
the fundamental group 
$\pi_1(X, x_0)$ and the fibre of the base point $p^{-1}(x_0)$, where $p: \widetilde{X}\to X$ is the endpoint projection map. 
Therefore, the fundamental group 
$\pi_1(X, x_0)$ as a subspace of the universal path space $\widetilde{X}$ inherits any topology from $\widetilde{X}$. 
Thus, the collection of sets with the form 
$ N(\langle \alpha\rangle, \mathcal{U}, W)\cap p^{-1}(x_0)$ is a basis for \emph{the lasso topology} on $\pi_1(X, x_0)$, which is denoted by 
$\pi_1^{lasso} (X, x_0)$ (see \cite[Definition 4.11]{Br12}).

5. \textbf{The Tau-topology}:  Brazas in \cite{B13} proved that there exists the finest topology on $\pi_{1}(X,x_{0})$ such that 
$\pi: \Omega(X,x_{0})\rightarrow \pi_{1}(X,x_{0})$ is continuous and $\pi_{1}(X,x_{0})$ is a topological group. 
The fundamental group $\pi_{1}(X,x_{0})$ with this topology is denoted by $\pi_1^{Tau}(X, x_0)$ (see \cite{B13}).

6. \textbf{The Spanier topology}:  As an example of the subgroup topology, the\emph{Spanier subgroup topology} \cite[p. 12]{W} 
was introduced using the collection 
of all Spanier subgroups $\pi(\mathcal{U}, x_0)$ of the fundamental group $\pi_1(X, x_0)$ as the neighbourhood family $ \Sigma^{Span} $. 
Recall that \cite[p. 81]{Span}, 
the Spanier subgroup determined by an open covering $\mathcal{U}$ of $X$ is the normal subgroup $\pi(\mathcal{U}, x_0)$ of 
$\pi_1(X, x_0)$ generated by the 
homotopy class of lollipops $\alpha * \beta * \alpha^{-1}$, where $\beta$ is a loop lying in an element of $U \in \mathcal{U}$ at $\alpha(1)$, 
and $\alpha$ is any 
path originated at $x_0$. The fundamental group equipped with the Spanier subgroup topology is denoted by $\pi_1 ^{Span}(X, x_0)$ 
(see \cite{A20,W}).

7. \textbf{The path Spanier topology}: Torabi et al. \cite[Section 3]{T} replaced open coverings with path open coverings of the space $ X $ 
in the definition of 
Spanier subgroups and introduced path Spanier subgroups by the same way. Recall that a path open covering $ \mathcal{V} $ of the path 
component of $ X $ involve $ x_0 $
 is the collection of open subsets $\lbrace V_{\alpha} \ \vert \ \alpha\in{P(X,x_{0})}\rbrace$ and the path Spanier subgroup 
 $\widetilde{\pi}(\mathcal{V}, x_{0})$ with respect to the path open covering $ \mathcal{V} $ is the subgroup of $ \pi_1(X, x_0)$ consists of 
 all homotopy classes 
 having representatives of the following type:
$$\prod_{j=1}^{n}\alpha_{j}\beta_{j}\alpha^{-1}_{j},$$
where $\alpha_{j}$'s are arbitrary path starting at $x_{0}$ and each $\beta_{j}$ is a loop inside of the open set $V_{\alpha_{j}}$ for 
all $j\in{\lbrace1,2,...,n\rbrace}$. 
Note that the path Spanier subgroup $\widetilde{\pi}(\mathcal{V}, x_{0})$ is not a normal subgroup, in general (see \cite[Example 3.7]{A20}).
If $ \mathcal{U} $ and $ \mathcal{V} $ are two path open covers of a space $ X $, the collection 
$ \mathcal{W}=\lbrace U_{\alpha} \cap V_{\alpha} \ \vert \ \forall \alpha\in{P(X,x_{0}), U_{\alpha} \in \mathcal{U} \ and \ V_{\alpha} 
\in \mathcal{V} }\rbrace $ 
is a refinement of both $ \mathcal{U} $ and $ \mathcal{V} $. Thus, 
$\widetilde{\pi}(\mathcal{W}, x_{0}) \leq  \widetilde{\pi}(\mathcal{U}, x_{0}) \cap \widetilde{\pi}(\mathcal{V}, x_{0}) $, 
which shows that the collection of all 
path Spanier subgroups of the fundamental group forms a neighbourhood family.
For a pointed topological space $ (X,x_0) $, let $ \Sigma^{pSpan} $ be the collection of all path Spanier subgroups of  $ \pi_1(X, x_0) $. 
The subgroup topology determined by $\Sigma^{pSpan} $ is called the \emph{path Spanier topology} which is denoted by 
$ \pi_1^{pSpan}(X, x_0)$ (see \cite{A20,T}).

8. \textbf{The generalized covering topology}: Brazas \cite{B15} introduced a generalized covering subgroup for a pointed topological space 
$ (X,x_0) $. 
Let $\Sigma^{gcov} $ be the collection of all generalized covering subgroups of  $\pi_1(X, x_0)$. The subgroup topology determined 
by $\Sigma^{gcov} $ is 
called the \emph{generalized covering topology} and denoted by $\pi_1^{gcov}(X, x_0)$ (see \cite{A16,A20,B15,FZ7}).

9. \textbf{The shape topology}: The first shape homotopy group of a pointed topological space $(X,x_0)$ is the inverse limit
$$ \check{\pi}_1(X,x_0)=\varprojlim (\pi_1(|N(\mathcal{U})|,U_0),p_{\mathcal{U}\mathcal{V}*},\Lambda),  $$
where the inverse system $\pi_1(|N(\mathcal{U})|,U_0),p_{\mathcal{U}\mathcal{V}*},\Lambda)$ of discrete groups is the first pro-homotopy group 
topologized with the usual inverse limit topology. The \emph{shape topology} on $\pi_1(X, x_0)$ is the initial topology with respect to the 
first shape homomorphism 
$\Psi_1: \pi_1(X, x_0)\rightarrow \check{\pi}_1(X,x_0)$.
Let $\pi_1^{sh}(X, x_0)$ denote $\pi_1(X, x_0)$ equipped with the shape topology group. Note that one can consider the shape topology 
on $\pi_1(X, x_0)$ as a 
subgroup topology with respect to all kernels of homomorphisms induced by maps $X\rightarrow K$ to simplicial complexes $K$ 
(see \cite{B13,BF14,BF24} for more details).

10. \textbf{The thick Spanier topology}:  In order to study the kernel of the first shape homomorphism $\Psi_1$, 
Brazs and Fabel \cite{BF14} introduced the 
thick Spanier subgroup of $X$ with respect to an open cover $\mathcal{U}$ of $X$, denoted by $\Pi^{sp}(\mathcal{U}, x_0)$, 
with a modification in the definition of 
Spanier subgroups, as a subgroup of $\pi_1(X, x_0)$ generated by the homotopy classes of $\alpha * \beta *\gamma* \alpha^{-1}$, 
where $\beta$ and $\gamma$ are 
paths lying in elements of $U_1,U_2 \in \mathcal{U}$, respectively, and $\alpha$ is any path originated at $x_0$. 
The collection of all thick Spanier subgroups 
$\Pi^{sp}(\mathcal{U}, x_0)$ of the fundamental group $\pi_1(X, x_0)$ forms a neighbourhood family $ \Sigma^{tSpan} $ 
(see \cite[Proposition 3.10]{BF14}). 
The fundamental group equipped with the thick Spanier subgroup topology is denoted by $\pi_1 ^{tSpan}(X, x_0)$. 
The thick Spanier subgroup of $X$, denoted by 
${\Pi }^{sp}(X,x_0)$, is the intersection of all thick Spanier subgroups relative to open covers of $X$ \cite[Definition 3.8]{BF14}. 
Note that thick Spanier subgroups 
$\Pi^{sp}(\mathcal{U}, x_0)$ and so $\Pi^{sp}(X, x_0)$ are normal subgroups of $\pi_1(X, x_0)$.


\section{ Some Results on $\pi_1^{\tau}$-Covering Maps}

In this section, first we recall some well-known notions in order to investigate some properties of $\pi_1^{\tau}$-covering maps 
(see \cite{Span, TMP}).

\begin{definition}
Let $p: E \to B$ be a continuous. Then we say that:
\begin{itemize}
\item[(i)] $p$ has the \textit{path lifting property} (pl for abbreviation) if for every path $\alpha : I \to B$, there exists (not necessarily unique) a path 
$\widetilde{\alpha}: I \to E$ such that $p \circ \widetilde{\alpha} = \alpha$.
\item[(ii)] $p$ has the \textit{unique lifting property} (ul) if for every connected, locally path connected space $(Y,y)$ and every continuous map 
$f: (Y,y) \to (B,b)$ with $\pi_1(f)\pi_1(Y,y) \subseteq \pi_1(p)\pi_1(E,e)$ for $e \in p^{-1}(b)$, there exists a unique continuous map $\widetilde{f}: (Y,y) \to (E,e)$ 
such that $p \circ \widetilde{f} = f$.
\item[(iii)] $p$ has the \textit{unique path lifting property} (upl) if for every two paths $\widetilde{\alpha}, \widetilde{\beta}: I \to E$ the equality 
$p \circ \widetilde{\alpha} = p \circ \widetilde{\beta}$ with $\widetilde{\alpha}(0) = \widetilde{\beta}(0)$ implies $\widetilde{\alpha} = \widetilde{\beta}$.
\item[(iv)] $p$ has the \textit{unique path homotopically lifting property} (uphl) if for every two paths $\widetilde{\alpha}, \widetilde{\beta}: I \to E$ the path homotopy 
$p \circ \widetilde{\alpha} \simeq p \circ \widetilde{\beta}$ rel $\dot{I}$ with $\widetilde{\alpha}(0) = \widetilde{\beta}(0)$ implies  
$\widetilde{\alpha} \simeq \widetilde{\beta}$  rel $\dot{I}$.
\item[(v)] A map $p: E \to B$ is said to have the \textit{homotopy lifting property with respect to a space $X$} if given maps
 $f^{\prime}: X \to E$ and $F: X \times I \to B$ such that $F(x,0) = p\circ f^{\prime}(x)$ for $x \in X$, there is a map 
 $F^{\prime}: X \times I \to E$ such that $F^{\prime}(x,0) = f^{\prime}(x)$ for $x \in X$ and $p \circ F^{\prime} = F$.
$$ \xymatrix{
X \times 0 \ar[r]^{f^{\prime}} \ar@{^{(}->}[d] & E \ar[d]^{p}\\
X \times I \ar@{.}^{F^{\prime}}[ur] \ar[r]^{F} & B
}$$
 A map $p: E \to B$ is called a \textit{fibration} if $p$ has the homotopy lifting property with respect to every space.
\item[(vi)] For any continuous map $p: E \to B$, there is a map $\overline{p}: E^I \to \overline{B}$ defined by 
$\overline{p}(\overline{\alpha}) = (\overline{\alpha}(0),p \circ \overline{\alpha})$ for $\overline{\alpha}: I \to E$, 
where $\overline{B} = \{(e,\alpha) \in E \times B^I | \alpha(0) = p(e)\}$. It is easy to see that $p$ has pl property if and only if there exists a map 
$\lambda : \overline{B} \to E^I$ which is a right inverse of $\overline{p}$.
A \textit{lifting function} for $p$ is a \textit{continuous} map $\lambda : \overline{B} \to E^I$ which is a right inverse of $\overline{p}$.  
Spanier \cite[Theorem 2.7.8]{Span} proved that a continuous map $p: E \to B$ is a fibration if and only if there exists a lifting function for $p$.
\end{itemize}
\end{definition}

It can be shown that every covering map is a fibration, but the converse is not true in general (see \cite{Span}). 
Also, a fibration $p: \widetilde{X} \to X$ with upl, where $X$ and $\widetilde{X}$ are connected and locally path connected spaces, is a covering map 
if and only if there is an open covering $\mathcal{U}$ of $X$ and a point $\widetilde{x}_0 \in \widetilde{X}$ such that 
$\pi (\mathcal{U},p(\widetilde{x}_0)) \subseteq\pi_1(p)\pi_1 (\widetilde{X},\widetilde{x}_0)$ (see \cite[Theorem 2.5.12]{Span}). 
The composition and product of two fibrations is also a fibration. Note that the property upl implies uphl property but the converse 
is not true in general. 
Also the properties upl and uphl are equivalent for fibrations (see \cite{TMP}). By \cite[Theorem 2.4.5]{Span} every fibration with upl has ul property 
and so is a generalized covering map. Clearly if a fibration with upl is a local homeomorphism then it is a semicovering map.

\begin{theorem}\label{3.2}
If $p_1: \widetilde{X}_1 \to X$ has pl, $p_2: \widetilde{X}_2 \to X$ has upl and $f: \widetilde{X}_1 \to \widetilde{X}_2$ 
with $p_2 \circ f = p_1$. Then $f$ has pl property.
\end{theorem}
\begin{proof}
Let $\alpha : I \to \widetilde{X}_2$ be a path, then $p_2 \circ \alpha$ is a path in $X$. Since $p_1$ has pl there exists a path 
$\widetilde{\alpha}: I \to \widetilde{X}_1$ such that $p_1 \circ \widetilde{\alpha} = p_2 \circ \alpha$. 
Thus $p_2 \circ f \circ \widetilde{\alpha} = p_2 \circ \alpha$ 
and since $p_2$ has upl so $f \circ \widetilde{\alpha} = \alpha$. Hence $f$ has pl property.
\end{proof}

\begin{theorem}\label{3.3}
Let $p: (\widetilde{X},\widetilde{x}_0) \to (X,x_0)$ be a continuous map. Then $\pi_1(p): \pi_1(\widetilde{X},\widetilde{x}_0) \to \pi_1(X,x_0)$ 
is a monomorphism if and only if $p$ has uphl.
\end{theorem}
\begin{proof}
Let $\pi_1(p)$ is a monomorphism and let $p \circ \widetilde{\alpha} \simeq p \circ \widetilde{\beta}$ rel $\dot{I}$, 
$\widetilde{\alpha}(0) = \widetilde{\beta}(0)$. 
Thus $[p \circ \widetilde{\alpha}] = [p \circ \widetilde{\beta}]$ i.e. $\pi_1(p)([\alpha]) = \pi_1(p)([\beta])$. Since $\pi_1(p)$ is a monomorphism so 
$[\alpha] = [\beta]$ i.e $\alpha \simeq \beta$ rel $\dot{I}$. Hence $p$ has uphl.
Conversely, let $p$ has uphl and let $\pi_1(p)([\alpha]) = \pi_1(p)([\beta])$. 
Thus $p \circ \widetilde{\alpha} \simeq p \circ \widetilde{\beta}$ rel $\dot{I}$ 
and $\widetilde{\alpha}(0) = \widetilde{x}_0 = \widetilde{\beta}(0)$ so $\alpha \simeq \beta$ rel $\dot{I}$ because $p$ has uphl. 
Therefore $\pi_1(p)$ is a monomorphism.
\end{proof}

Some researchers proved that if $p: (\widetilde{X},\widetilde{x}_0) \to (X,x_0)$ is a continuous map, 
then $\pi_1^{Tau}$ \cite{B13}, $\pi_1^{qtop}$ \cite[Proposition 3.4]{Biss} , 
$\pi_1^{wh}(p)$ \cite[Lemma 3.1]{A16} and $\pi_1^{lasso}(p)$ \cite[Proposition 4.1]{Br8} preserve continuity.
In the following, we want to investigate this property for subgroup topology.

\begin{theorem}\label{3.4}
Let $p: (\widetilde{X},\widetilde{x}_0) \to (X,x_0)$ be a continuous map and $K \leq \pi_1(p)^{-1}(H)$ where $H \leq \pi_1(X,x_0)$. 
Then $\pi_1(p): \pi_1^K(\widetilde{X},\widetilde{x}_0) \to \pi_1^H(X,x_0)$ is continuous.
\end{theorem}
\begin{proof}
For continuity of $\pi_1(p)$ it is enough to prove that $\pi_1(p)^{-1}(gH)$ is open in 
$\pi_1^K(\widetilde{X},\widetilde{x}_0)$ for each $g \in \pi_1(X,x_0)$. 
It is clear that $\pi_1(p)^{-1}(gH) = \phi$ or $\pi_1(p)^{-1}(gH) = \cup y_j\pi_1(p)^{-1}(H)$ where $\pi_1(p)(y_j) = g$. 
Since $K \leq \pi_1(p)^{-1}(H)$ so 
$\pi_1(p)^{-1}(gH) = \cup y_j\pi_1(p)^{-1}(H)$ is open in $\pi_1^K(\widetilde{X},\widetilde{x}_0)$. Thus $\pi_1(p)$ is continuous.
\end{proof}

Note that if $p: (\widetilde{X},\widetilde{x}_0) \to (X,x_0)$ is a continuous map and $H = \pi^{sp}(X,x_0)$, then 
$\pi^{sp}(\widetilde{X},\widetilde{x}_0) \subseteq \pi_1(p)^{-1}(H)$. Hence by \ref{3.4} 
$\pi_1(p): \pi_1^{\pi^{sp}(\widetilde{X},\widetilde{x}_0)}(\widetilde{X},\widetilde{x}_0) \to \pi_1^{\pi^{sp}(X,x_0)}(X,x_0)$ is continuous.

Similar to the category of fibration maps over $X$, $Fib(X)$ \cite{Span} and the category of covering, semicovering and generalized covering maps, 
$COV(X), SCOV(X), GCOV(X)$ \cite{A16} one can define the category of $\pi_1^{\tau}$-covering maps over $X$. In the following theorem,
 we show that when a morphism of this category is also a $\pi_1^{\tau}$-covering map.

\begin{theorem}\label{3.5}
Let $p_1: (\widetilde{X}_1,\widetilde{x}_1) \to (X,x_0)$ be a $\pi_1^{\tau}$-covering map, 
$p_2: (\widetilde{X}_2,\widetilde{x}_2) \to (X,x_0)$ be a 
$\pi_1^{\tau}$-covering map with uphl property, and $f: \widetilde{X}_1 \to \widetilde{X}_2$ be a continuous map with $p_2 \circ f = p_1$. 
If $\pi_1^{\tau}(p_2)$ is continuous, then $f$ is a $\pi_1^{\tau}$-covering map.
\end{theorem}
\begin{proof}
Since $p_2$ has uphl so $\pi_1(p_2)$ is a monomorphism. Thus $\mathrm{Im}\pi_1(f) =$\\  
$ \pi_1(p_2)^{-1} (\pi_1(p_2) (\mathrm{Im}\pi_1(f)))$. 
On the other hand, $\pi_1(p_2)(\mathrm{Im}\pi_1(f)) = \mathrm{Im}\pi_1(p_1)$ which is open in 
$\pi_1^{\tau}(X,x_0)$ because $p_1$ is a $\pi_1^{\tau}$-covering map. Since $\pi_1^{\tau}(p_2)$ 
is continuous so $\mathrm{Im}\pi_1(f)$ is open in 
$\pi_1^{\tau}(\widetilde{X}_2,\widetilde{x}_2)$. By \ref{3.2} $f$ has pl property. Hence $f$ is a $\pi_1^{\tau}$-covering map.
\end{proof}

\begin{theorem}\label{3.6}
Let $p_1: (\widetilde{X}_1,\widetilde{x}_1) \to (X_1,x_1)$ and $p_2: (\widetilde{X}_2,\widetilde{x}_2) \to (X_2,x_2)$ 
be two $\pi_1^{\tau}$-covering maps. 
If $\psi : \pi_1^{\tau}(X_1 \times X_2,(x_1,x_2)) \to \pi_1^{\tau}(X_1,x_1) \times \pi_1^{\tau}(X_2,x_2)$ ; 
$\psi ([f,g]) = ([f],[g])$ is homeomorphism, 
then $p_1 \times p_2: (\widetilde{X}_1 \times \widetilde{X}_2,(\widetilde{x}_1,\widetilde{x}_2)) \to (X_1 \times X_2,(x_1,x_2))$ 
is a $\pi_1^{\tau}$-covering map.
\end{theorem}
\begin{proof}
Let $\alpha : I \to X_1 \times X_2$ be a path then $\alpha = (\theta_1 \circ \alpha,\theta_2 \circ \alpha)$, 
where $\theta_i$ are the projection maps for $i = 1,2$. 
Since $p_i$ has pl property, there is a path $\widetilde{\alpha}_i: I \to \widetilde{X}_i$ 
with $p_i \circ \widetilde{\alpha}_i = \theta_i \circ \alpha$ for $i = 1,2$. 
Put $\widetilde{\alpha} = (\widetilde{\alpha}_1,\widetilde{\alpha}_2)$, then $(p_1 \times p_2) \circ \widetilde{\alpha} = \alpha$. 
Thus $p_1 \times p_2$ has pl property.
Since $p_i$ is a $\pi_1^{\tau}$-covering map, $\mathrm{Im}(\pi_1(p_i))$ is open in $\pi_1^{\tau}(X_i,x_i)$ for $i = 1,2$.
Since  $\psi (\mathrm{Im}(\pi_1(p_1 \times p_2))) = \mathrm{Im}(\pi_1(p_1)) \times \mathrm{Im}(\pi_1(p_2))$ and $\psi$ is a homeomorphism, 
$\mathrm{Im}(\pi_1(p_1 \times p_2))$ is open in 
$\pi_1^{\tau}(X_1 \times X_2,(x_1,x_2))$. Hence $p_1 \times p_2$ is a $\pi_1^{\tau}$-covering map.
\end{proof}

Note that  $\pi_1^{qtop}$ (under some conditions) 
\cite[Lemma 41]{BF15}, $\pi_1^{Tau}$ \cite[Proposition 3.19]{B13}, 
$\pi_1^{lasso}$ \cite[Remark 4.15]{PA19} and $\pi_1^H$  (under some conditions)
\cite[Theorem 2.6]{SM} preserve finite products.

In the following theorem, we show that the well-known notions, covering, semicovering, 
generalized covering and fibration maps are special cases of this new notion 
$\pi_1^{\tau}$-covering, for appropriate topologies on $\pi_1(X, x_0)$. 
We recall that a continuous map $p: (\widetilde{X},\widetilde{x}_0) \to (X,x_0)$ 
is a semicovering map if and only if it is a local homeomorphism with pl and upl properties (see \cite[Theorem 2.4]{KMT}). Also, a continuous map 
$p: (\widetilde{X},\widetilde{x}_0) \to (X,x_0)$ is called a generalized covering map if $p$ has ul property (see \cite[Definition 2.1]{A16}).
\begin{theorem}\label{3.7}
Let $p: (\widetilde{X},\widetilde{x}_0) \to (X,x_0)$ be a continuous map.
\begin{itemize}
\item[(i)] If $p$ is a covering map, then $p$ is a $\pi_1^{Span}$-covering map. 
Moreover, if $p$ is a $\pi_1^{Span}$-covering map with unique lifting property, 
then $p$ is a covering map.
\item[(ii)] If $p$ is a semicovering map, then $p$ is a $\pi_1^{pSpan}$-covering map. 
Moreover, if $p$ is a $\pi_1^{pSpan}$-covering map with unique lifting property, 
then $p$ is a semicovering map.
\item[(iii)] If $p$ is a generalized covering map, then $p$ is a $\pi_1^{gcov}$-covering map.
\item[(iv)] If $p$ is a fibration map, then $p$ is a $\pi_1^{\pi^{fib}(X,x_0)}$-covering map, where $\pi^{fib}(X,x_0)$ is the intersection of all 
$\mathrm{Im}\pi_1(q)$ for any fibration map $q:(E,e_0)\rightarrow (X,x_0)$.
\end{itemize}
\end{theorem}

\begin{proof}
$(i)$ Since $p$ is a covering map, by \cite[Section 2.5 Theorem 12]{Span}, there is an open covering $\mathcal{U}$ of $X$ such that 
$\pi(\mathcal{U},x_0) \subseteq \pi_1(p)\pi_1(\widetilde{X},\widetilde{x}_0)$. By the Spanier topology on  $\pi_1(X, x_0)$, $\pi(\mathcal{U},x_0)$ 
is an open subgroup of $\pi_1^{Span}(X,x_0)$.   Since $\pi_1(p)\pi_1(\widetilde{X},\widetilde{x}_0)$ is a union of some cosets of 
$\pi(\mathcal{U},x_0)$, it is open in $\pi_1^{Span}(X,x_0)$. 
Note that every covering map has unique path lifting property, hence $p$ is a $\pi_1^{Span}$-covering map.

Conversely, let $p: (\widetilde{X},\widetilde{x}_0) \to (X,x_0)$ be a $\pi_1^{Span}$-covering map with unique lifting property. Put 
$H = \pi_1(p)\pi_1(\widetilde{X},\widetilde{x}_0)$, then $H$ is open in $\pi_1^{Span}(X,x_0)$. Since $H$ is a subgroup, there exists an open cover 
$\mathcal{U}$ of $X$ such that $\pi(\mathcal{U},x_0) \subseteq H$. By \cite[Theorem 2.5.13]{Span}, $p_H : \widetilde{X}_H \to X$ 
is a covering map and so by lifting criterion there exists a continuous map $\varphi : \widetilde{X} \to \widetilde{X}_H$ 
such that $p_H \circ \varphi = p$. 
Also, since $p$ has unique lifting property there exists a continuous map $\psi : \widetilde{X}_H \to \widetilde{X}$ 
such that $p \circ \psi = p_H$. 
By uniqueness $\psi \circ \varphi = \varphi \circ \psi =1$ i.e. $\varphi$ is a homeomorphism. 
Thus $p$ and $p_H$ are equivalent and hence $p$ is a covering map since 
$p_H$ is covering map.

$(ii)$ Since $p$ is a semicovering map, by \cite[Theorem 4.1]{T} there exists a path open covering $\mathcal{V}$ of $X$ such that 
$\widetilde{\pi}(\mathcal{V},x_0) \subseteq \pi_1(p)\pi_1(\widetilde{X},\widetilde{x}_0)$. 
By the path Spanier topology on  $\pi_1(X, x_0)$, $\pi(\mathcal{V},x_0)$ 
is an open subgroup of $\pi_1^{pSpan}(X,x_0)$.   
Since $\pi_1(p)\pi_1(\widetilde{X},\widetilde{x}_0)$ is a union of some cosets of $\pi(\mathcal{V},x_0)$, 
it is open in $\pi_1^{pSpan}(X,x_0)$. Note that every semicovering map has path lifting property, hence $p$ is a $\pi_1^{pSpan}$-covering map.

Conversely, let $p: (\widetilde{X},\widetilde{x}_0) \to (X,x_0)$ be a $\pi_1^{pSpan}$-covering map with unique lifting property. Put 
$H = \pi_1(p)\pi_1(\widetilde{X},\widetilde{x}_0)$, then $H$ is open in $\pi_1^{pSpan}(X,x_0)$. 
Since $H$ is a subgroup, there exists a path open cover 
$\mathcal{V}$ of $X$ such that $\widetilde{\pi}(\mathcal{V},x_0) \subseteq H$. By \cite[Theorem 4.1]{T}, $p_H : \widetilde{X}_H \to X$ 
is a semicovering map. Since every semicovering map is a generalized covering map and so has ul property, 
there exists a continuous map $\varphi : \widetilde{X} \to \widetilde{X}_H$ such that 
$p_H \circ \varphi = p$. 
Also, since $p$ has unique lifting property there exists a continuous map $\psi : \widetilde{X}_H \to \widetilde{X}$ such that $p \circ \psi = p_H$. 
By uniqueness $\psi \circ \varphi = \varphi \circ \psi =1$ i.e. $\varphi$ is a homeomorphism. 
Thus $p$ and $p_H$ are equivalent and hence $p$ is also a semicovering map 
since $p_H$ is semicovering map.

$(iii)$ Since $p$ is a generalized covering map, then $\pi_1(p)\pi_1(\widetilde{X},\widetilde{x}_0)$ 
is a generalized covering subgroup for $(X,x_0)$ (see \cite{B15}). 
By the definition of the generalized covering topology on $\pi_1(X, x_0)$, $\pi_1(p)\pi_1(\widetilde{X},\widetilde{x}_0)$ 
is open in $\pi_1^{gcov}(X,x_0)$. 
Also, we know that every generalized covering map has path lifting property. Therefore $p$ is a $\pi_1^{gcov}$-covering map. 

$(iv)$ Note that every fibration map has pl property.
\end{proof}

Brazas in \cite[Lemma 5.10]{B15} proved that every generalized covering map $p:(\widehat{X},\widehat{x}) \to (X,x_0)$ is equivalent to 
$p_H:(\widetilde{X}_H,\widetilde{x}_0) \to (X,x_0)$ where $H = \pi_1(p)\pi_1(\widehat{X},\widehat{x})$.
Abdullahi et. al in \cite[Corollary 3.10]{A16} proved that if $H$ is a generalized covering subgroup, 
then $H$ is a closed subgroup of $\pi_1^{wh}(X,x_0)$. 
Also in \cite[Proposition 3.2]{A16} it is proved that $\pi_1^{wh}(X,x_0)$ is a homogenous space consequently 
if $H$ is a closed subgroup of $\pi_1^{wh}(X,x_0)$ 
and $[\pi_1(X,x_0):H] < \infty$, then $H$ is also open. By these results we have the following relationship between 
$\pi_1^{wh}$-covering maps and generalized covering maps.

\begin{theorem}\label{3.8}
Let $p_H:(\widetilde{X}_H,\widetilde{x}_0) \to (X,x_0)$ be a generalized covering map. 
If $[\pi_1(X,x_0):H] < \infty$, then $p_H$ is a $\pi_1^{wh}$-covering map.
\end{theorem}
\begin{proof}
Since $p_H$ is a generalized covering map, $H$ is a generalized covering subgroup. 
Then by \cite[Corollary 3.10]{A16} $H$ is closed in $\pi_1^{wh}(X,x_0)$ 
and so $H$ is open in $\pi_1^{wh}(X,x_0)$  because  $[\pi_1(X,x_0):H] < \infty$. 
Thus $p_H:(\widetilde{X}_H,\widetilde{x}_0) \to (X,x_0)$ is a $\pi_1^{wh}$-covering map.
\end{proof}

\section{ Comparison of $\pi_1^{\tau}$-Covering Maps with respect to Topologies on the Fundamental Group}

In this section, we are going to compare $\pi_1^{\tau}$-covering maps of a space $X$ 
for several famous topologies on the fundamental group of $\pi_1(X,x_0)$ 
 which have been reviewed in Section 2. First, we review some results obtained by researchers for comparison of the above topologies.
 
 Fischer and Zastrow \cite[Lemma 2.1]{FZ7} showed that the whisker topology is finer than the $ qtop $-topology on the universal path space 
 $ \widetilde{X} $ 
 for any space $ X $. Clearly, the result will hold for the fundamental group $\pi_{1}(X,x_{0})$ as a subspace of $ \widetilde{X} $. 
 By considering the definitions of $\pi_{1}^{qtop}(X,x_0)$ and $\pi_1^{Tau}(X, x_0)$ it is easy to see that 
 $\pi_{1}^{qtop}(X,x_0)$ is finer than $\pi_1^{Tau}(X, x_0)$ 
 (see \cite{B13}).
 It is proved in \cite{A20} that the lasso topology on the fundamental group $\pi_1(X, x_0)$ coincides with the Spanier subgroup topology.
Since every Spanier subgroup of the fundamental group $\pi_1(X, x_0)$ is also a path Spanier subgroup, for any pointed space $ (X, x_0) $ 
the path Spanier topology on the fundamental group, $\pi_1^{pSpan}(X, x_0)$, is finer than the Spanier topology, $\pi_1^{Span}(X, x_0)$.
Also using  \cite[Proposition 3.16]{B13} the authors of \cite{A20} showed that
if $X$ is a locally path connected space, then $\pi_1^{pSpan}(X, x_0)$ is coarser than $\pi_1^{\tau}(X, x_0)$ (see \cite[Corollary 3.15]{A20}).
 It is shown in \cite[Proposition 3.2]{BF24} that the shape topology of $\pi_1^{sh}(X, x_0)$ is coarser than that of $\pi_1^{\tau}(X, x_0)$.
Note that by \cite[Proposition 3.5]{BF14} one can show that $\pi_1^{Span}(X, x_0)$ is finer than $\pi_1^{tSpan}(X, x_0)$. 
Also by \cite[Proposition 5.8]{BF14} $\pi_1^{tSpan}(X, x_0)$ is finer than $\pi_1^{sh}(X, x_0)$.

One can summarize the relationship between the mentioned topologies on the fundamental group by the above statements in the following chain 
(note that we use the symbol $\preccurlyeq$ to show the finer topology on a group. For example, $G^{\tau_1}\preccurlyeq  G^{\tau_2}$ means that 
$\tau_2$ is finer than $\tau_1$ and $G^{\tau_1}\prec  G^{\tau_2}$ means that $\tau_2$ is strictly finer than $\tau_1$). 
$$ \pi_1^{sh}(X, x_0) \preccurlyeq \pi_1^{tSpan}(X,x_0) \preccurlyeq \pi_1^{lasso} (X, x_0)=\pi_1^{Span}(X, x_0) \preccurlyeq$$ 
$$\pi_1^{pSpan}(X, x_0) \preccurlyeq \pi_1^{\tau}(X, x_0) \preccurlyeq \pi_1^{qtop}(X, x_0)  \preccurlyeq \pi_1^{wh}(X, x_0).$$

There are some examples to show that most of the above topologies are strictly finer than the previous one. As a good example, 
consider the {\it Hawaiian Earring} space, $\mathbb{HE}$. Using the results of \cite{A20, B13,F9,J} 
there exists the following chain of strictly finer topologies on the fundamental group of $\mathbb{HE}$: 
$$\pi_1^{Span}(\mathbb{HE}, 0) \prec \pi_1^{pSpan}(\mathbb{HE}, 0),$$ 
$$\pi_1^{sh}(\mathbb{HE}, 0)\prec \pi_1^{\tau}(\mathbb{HE}, 0)\prec \pi_1^{qtop}(\mathbb{HE}, 0)\prec \pi_1^{wh}(\mathbb{HE}, 0).$$

For the generalized covering topology, it is proved in \cite[Proposition 3.24]{A20} that 
$$\pi_1^{qtop}(X, x_0)  \preccurlyeq \pi_1^{gcov}(X, x_0),$$
when $X$ is a connected and locally path connected. Note that $\pi_1^{wh}(\mathbb{HE}, 0)  \prec \pi_1^{gcov}(\mathbb{HE}, 0)$ and 
$\pi_1^{gcov}(\mathbb{HA}, b)  \prec \pi_1^{wh}(\mathbb{HA}, b)$, where $\mathbb{HA}$ is the {\it Harmonic Archipelago} 
and $b\neq 0$ (see \cite{A16,A20,FZ7}). Therefore, the whisker topology and the generalized covering topology can not be comparable, in general.
Moreover, if $X$ is locally path connected, paracompact and Hausdorff space , 
then by \cite[Theorem 7.6]{BF14} one can see that the equality $\pi_1^{sh}(X, x_0)=\pi_1^{tSpan}(X,x_0)=\pi_1^{Span}(X, x_0)$ holds. 

We can compare most of the above topologies on $\pi _1(X,x_0)$ in view of the subgroup topology for some famous subgroups in the following chain of subgroups 
of the fundamental group ${\pi }_1(X,x_0)$ (see \cite{A16}),
$$\{e\}\leq {\pi }^s(X,x_0)\leq {\pi }^{sg}(X,x_0)\leq \pi ^{gc}(X,x_0)\leq {\widetilde{\pi }}^{sp}(X,x_0)\leq {\pi }^{sp}(X,x_0)\leq 
\pi_1(p)\pi_1(\widetilde{X},\widetilde{x}_0), $$
where ${\pi }^s(X,x_0)$ is the subgroup of all small loops at $x_0$ \cite{V}, ${\pi }^{sg}(X,x_0)$ is the subgroup of all small generated loops \cite{V},
$\pi^{gc}(X,x_0)$ is the intersection of generalized covering subgroups \cite{A16}, ${\pi }^{sp}(X,x_0)$ is the Spanier group of $X$, 
the intersection of the Spanier subgroups relative to open covers of $X$ \cite[Definition 2.3]{FR},  ${\widetilde{\pi }}^{sp}(X,x_0)$ is the path Spanier group, 
i.e, the intersection of all path Spanier subgroups ${\widetilde{\pi }}(\mathcal{V},x_0)$, where $\mathcal{V}$ is a path open cover of $X$ \cite[Section 3]{T},    
and $\pi_1(p)\pi_1(\widetilde{X},\widetilde{x}_0)\cong \pi_1(\widetilde{X},\widetilde{x}_0) $ is the image of the induced homomorphism of an arbitrary covering map 
$p:(\widetilde{X},\widetilde{x}_0)\rightarrow (X,x_0)$. 

Recently, the authors in \cite{SM} by presenting some more results on the subject, have sum up the comparison of various famous topologies on the fundamental group 
in the following diagram (note that $A \longrightarrow B$ means that $A\preccurlyeq B$).

$$ \scalebox{0.7} { \xymatrix{
&\pi_1^{\pi^s(X,x_0)}(X,x_0)&\\
&&\pi_1^{\pi^{sg}(X,x_0)}(X,x_0) \ar[lu]^{(2)}\\
&\pi_1^{wh}(X,x_0) \ar[uu]^{(1)} \ar@{.}[rd]^{(4)}&\\
\pi_1^{qtop}(X,x_0) \ar[ru]^{(5)} \ar[rr]^{(6)}&&\pi_1^{gcov}(X,x_0)=\pi_1^{\pi^{gc}(X,x_0)}(X,x_0) \ar[uu]^{(3)} \ar@{.}[lu] \\
&&&\\
\pi_1^{Tau}(X,x_0) \ar[uu]^{(8)}&&\pi_1^{\widetilde{\pi}^{sp}(X,x_0)}(X,x_0) \ar[uu]^{(7)}\\
&\pi_1^{pSpan}(X,x_0) \ar[lu]^{(9)} \ar[ru]^{(10)}&\\
&&&\\
&\pi_1^{\pi^{sp}(X,x_0)}(X,x_0) \ar[uu]^{(11)}&\\
&&&\\
&\pi_1^{lasso}(X,x_0)=\pi_1^{Span}(X,x_0)\ar[uu]^{(12)}&\pi_1^{\prod^{sp}(X,x_0)}(X,x_0)\ar[luu]^{(15)}\\
&&&\\
&\pi_1^{tSpan}(X,x_0) \ar[uu]^{(13)}\ar[ruu]^{(16)}&\\
&&&\\
&\pi_1^{sh}(X,x_0)\ar[uu]^{(14)}&\\
 }}$$
 
$$ \mathbf{Diagram\ 1}$$

Finally, we summarize the comparison of $\pi_1^{\tau}$-covering maps of a space $X$ for several famous topologies on the fundamental group of 
$\pi_1(X,x_0)$ in the following diagram. Note that $A \longrightarrow B$ means that every covering maps of type $A$ is of type $B$. 
Also, $A$-covering $=$ $B$-covering means that every covering maps of type $A$ is of type $B$ and vice versa.

$$ \scalebox{0.7}{\xymatrix{
\pi_1^{\pi^s(X,x_0)}\text{- covering} &  & \\ 
\pi_1^{\pi^{sg}(X,x_0)}\text{- covering}\ar[u]_{(1)} & & \\
\pi_1^{\pi^{gc}(X,x_0)}\text{-covering} \overset{(4)}{=} \pi_1^{gcov}\text{- covering} \ar[u]^{(3)} \ar@{.}[rr]^{(5)} & 
& \pi_1^{wh}\text{- covering}\ar[uull]^{(2)}\ar@{.}[dd]^{(10)} \\
\pi_1^{\widetilde {\pi}^{sp}(X,x_0)}\text{- covering}\ar[u]^{(6)} & & \\
\pi_1^{pSpan}\text{- covering} \ar[u]^{(11)}\ar[r]_{(12)} &\pi_1^{qtop}\text{- covering} \overset{(13)}{=}  \pi_1^{Tau}\text{- covering} \ar[uul]^{(7)} \ar[uur]^{(9)} 
& \text{generalized covering map} \ar[uull]^{(8)} \\
\pi_1^{\pi^{sp}(X,x_0)}\text{- covering}\ar[u]^{(14)} \ar[r]^{(15)} & \pi_1^{\pi^{fib}(X,x_0)}\text{-covering} & \text{semicovering map}\ar[ull]_{(16)}\ar[u]_{(17)} \\
\pi_1^{lasso}\text{- covering} \overset{(20)}{=} \pi_1^{Span}\text{- covering} \ar[u]^{(18)} & \text{fibration} \ar[u]^{(19)} & \\
\pi_1^{tSpan}\text{- covering} \ar[u]^{(21)} & \text{fibration} + \text{upl} \ar[u]^{(22)} \ar[uuur]^{(23)} \ar@{.}[uur]^{(24)} & \\
\pi_1^{sh}\text{- covering} \ar[u]^{(25)} &  \text{covering map} \ar[uuur]_{(28)} \ar[uul]_{(26)} \ar[u]^{(27)} &\\
 \pi_1^{tSpan}\text{- covering} \ar[r]^{(29)} & \pi_1^{\Pi^{sp}}\text{- covering} \ar[r]^{(30)} & \pi_1^{\pi^{sp}}\text{- covering}
}}$$
$$ \mathbf{Diagram\ 2}$$

In the following, according to the enumeration in the above diagram, we give references and complementary notes for each arrow.
\begin{itemize}
\item[(1)] By (2) in Diagram 1. Clearly, if $\pi^s(X,x_0) = \pi^{sg}(X,x_0)$ then 
$\pi_1^{\pi^s(X,x_0)}\text{- covering map}$ $=$ $\pi_1^{\pi^{sg}(X,x_0)}\text{- covering map}$.
By \cite[example 3.12]{A16} $\pi^s(\mathbb{HA},b) = 1$ but $\pi^{sg}(\mathbb{HA},b) = \pi_1(\mathbb{HA},b)$, 
where $b \neq 0$. Consider $(X,x_0) = (\mathbb{HA},b)$ and $p_e: (\widetilde{X}_e,\widetilde{x}) \to (X,x_0)$, 
then $\mathrm{Im}\pi_1(p_e)$ is open in $\pi_1^{\pi^s(\mathbb{HA},b)}(X,x_0)$ since $\pi^s(\mathbb{HA},b) = 1$. 
Also $p_e$ has path lifting property thus $p_e$ is a $\pi_1^{\pi^s(\mathbb{HA},b)}$-covering. We conjecture that $\mathrm{Im}\pi_1(p_e) \neq \pi^{sg}(\mathbb{HA},b)$,
if so, then $p_e$ is not a $\pi_1^{\pi^{sg}(\mathbb{HA},b)}$-covering because $\pi_1^{\pi^{sg}(\mathbb{HA},b)}(X,x_0)$ is indiscrete 
(we are interested in fining out about the conjecture).
\item[(2)] By (1) in Diagram 1. The equality holds if $X$ is semilocally $\pi^s(X,x_0)$-connected at $x_0$ because $\pi_1^{wh}(X,x_0) = \pi_1^{\pi^s(X,x_0)}(X,x_0)$ 
(see \cite[Proposition 4.21]{Br12}). Consider $X = \mathbb{HE}$, then $\pi^s(\mathbb{HE},0) = 1$ and so $p_e$ is a $\pi_1^{\pi^s}$-covering map. We guess that 
$Im \pi_1(p_e)$ is not open in $\pi_1^{wh}(\mathbb{HE},0)$, if so, then $p_e$ is not a $\pi_1^{wh}$-covering map. 
\item[(3)] By (3) in Diagram 1. The equality holds if $\pi^{sg}(X,x_0) = \pi^{gc}(X,x_0)$ because $\pi_1^{gcov}(X,x_0) = \pi_1^{\pi^{sg}(X,x_0)}(X,x_0)$. 
Consider $X =RX$ in \cite[Definition 7]{FZ13}, then $\pi^{sg}(X,x_0) = 1$ but $\pi^{gc}(X,x_0) \neq 1$. Thus $p_e$ is a $\pi_1^{\pi^{sg}}$-covering map. 
We guess that $\mathrm{Im} \pi_1(p_e)$ is not open in $\pi_1^{\pi^{gc}(X,x_0)}(X,x_0)$, if so, then $p_e$ is not a $\pi_1^{\pi^{gc}}$-covering map.
\item[(4)] Since $\pi_1^{gcov}(X,x_0) = \pi_1^{\pi^{gc}(X,x_0)}(X,x_0)$ (see \cite[Theorem 3.3]{SM}).
\item[(5)] It is known that $\pi_1^{wh}(X,x_0)$ and $\pi_1^{gcov}(X,x_0)$ are not comparable in general (see \cite[Example 3.25 and Example 3.26]{A16}) 
and so $\pi_1^{wh}$-covering and $\pi_1^{gcov}$-covering are not comparable in general. If $X$ is connected, locally path connected and semilocally 
$\pi^{gc}(X,x_0)$-connected, then $\pi_1^{gcov}(X,x_0) \preccurlyeq \pi_1^{wh}(X,x_0)$ (see \cite[Corollary 3.4.2]{SM}) and so every $\pi_1^{gcov}\text{-covering}$ 
is a $\pi_1^{wh}\text{-covering}$. The equality holds for any semilocally simply connected space because $\pi_1^{gcov}(X,x_0) = \pi_1^{wh}(X,x_0)$.
\item[(6)] By (7) in Diagram 1. The equality holds if $\widetilde{\pi}^{sp}(X,x_0) = \pi^{gc}(X,x_0)$.
\item[(7)] By (6) in Diagram 1. The equality holds if $X$ is a connected, locally path connected and semilocally  $\pi^{gc}(X,x_0)$-connected space because 
$\pi_1^{qtop}(X,x_0) = \pi_1^{gcov}(X,x_0)$ (see \cite[Corollary 3.4.2]{SM}). The universal path space of $\mathbb{HE}$ is generalized covering space and so $p_e$ 
is a generalized covering map but it is not $\pi_1^{qtop}$-covering map because $\mathrm{Im}\pi_1(p_e) = 1$ and it is not open in $\pi_1^{qtop}(\mathbb{HE},0)$ 
since $\pi_1^{qtop}(\mathbb{HE},0)$ is not discrete (see \cite[Example 3.25]{A16}).
\item[(8)] By Theorem \ref{3.7}.
\item[(9)] By (5) in Diagram 1. The equality holds if $X$ is locally path connected and SLT at $x_0$ because $\pi_1^{qtop}(X,x_0) = \pi_1^{wh}(X,x_0)$ 
(see \cite[Corollary 3.3]{ PA17}).
\item[(10)] By Theorem \ref{3.8}.
\item[(11)] By (10) in Diagram 1. The equality holds if $\widetilde{\pi}^{sp}(X,x_0)$ is a semicovering subgroup.
\item[(12)] By (9) in Diagram 1. The equality holds if X is locally path connected and semilocally small generated.
\item[(13)] By \cite[Proposition 3.16]{B13} $\pi_1^{qtop}(X,x_0)$ and $\pi_1^{Tau}(X,x_0)$ have the same open subgroups thus every 
$\pi_1^{qtop}\text{-covering}$ is a $\pi_1^{Tau}\text{-covering}$ and vice versa.
\item[(14)] It is proved in \cite[Theorem 3.2]{SM} $\pi_1^{\pi^{sp}(X,x_0)}(X,x_0) \preccurlyeq \pi_1^{pSpan}(X,x_0)$ if and only if $\pi^{sp}(X,x_0)$ 
is a semicovering subgroup. Therefore, if $\pi^{sp}(X,x_0)$ is a semicovering subgroup, then every $\pi_1^{\pi^{sp}(X,x_0)}$- covering map is a 
$\pi_1^{pSpan}$-covering map. The equality holds if $\pi^{sp}(X,x_0) = \widetilde{\pi}^{sp}(X,x_0)$.
\item[(15)] Since every covering map is a fibration, $\pi^{fib}(X,x_0) \subseteq \pi^{sp}(X,x_0)$. The equality holds if $\pi^{fib}(X,x_0) = \pi^{sp}(X,x_0)$.
\item[(16)] By Theorem \ref{3.7}.
\item[(17)] It is known that every semicovering map is a generalized covering map. For the converse, It is proved in \cite[Prpposition 2.10]{A16} that if $X$ 
is connected and locally path connected, then a generalized covering map $p: (\widetilde{X},\widetilde{x}_0) \to (X,x_0)$ is a semicovering map if and only if $p$ 
is locally injective . Note that there exists an example of a generalized covering map which is not semicovering map (see \cite[Example 4.15]{FZ7}.
\item[(18)] By (12) in Diagram 1. The equality holds if $\pi^{sp}(X,x_0)$  is a covering subgroup. Put $H = \pi^{sp}(\mathbb{HE},0)$, then $p_H$ is a 
$\pi_1^{\pi^{sp}}$-covering map but it is not a $\pi_1^{Span}$-covering map because if $p_H$ is a $\pi_1^{Span}$-covering map, 
then it is universal covering map but we know that $\mathbb{HE}$ does not have universal covering map (see \cite[Corollary 3.7]{PTM17}).
\item[(19)] By definition of $\pi^{fib}(X,x_0)$.
\item[(20)] Since $\pi_1^{lasso}(X,x_0) = \pi_1^{Span}(X,x_0)$ (see \cite[Proposition 3.5]{A16}).
\item[(21)] By (13) in Diagram 1. The equality holds if $X$ is locally path connected, paracompact and Hausdorff (see \cite[Theorem 7.6]{BF14}).
\item[(22)] It is clear.
\item[(23)] By Theorem 2.4.5 in \cite{Span}.
\item[(24)] If a fibration with upl is a local homeomorphism, then by \cite[Theorem 2.4.5]{Span} it is a semicovering map.
\item[(25)] By (14) in Diagram 1. The equality holds if $X$ is locally path connected, paracompact and Hausdorff (see \cite[Theorem 7.6]{BF14}).
\item[(26)] By Theorem \ref{3.7}.
\item[(27)] By \cite[Theorem 2.2.3]{Span}. The equality holds if there is an open covering $\mathcal{U}$ of $X$ such that 
$\pi(\mathcal{U},x_0) \subseteq \pi_1(p)(\widetilde{X},\widetilde{x}_0)$ for a fibration $p : (\widetilde{X},\widetilde{x}_0) \to (X,x_0)$ 
with upl \cite[Theorem 2.5.13]{Span}. Note that there is an example of a fibraion with upl which is not a covering map in \cite[Example 2.2.9]{Span}.
\item[(28)] It is known that every covering map is a semicovering map. By \cite[Corollary 4.7]{A16} for a connected and locally path connected space 
$X$, GCOV$(X)=$COV($X$) if and only if  $X$ is semilocally $\pi_1^{gc}(X,x_0)$-connected. Brazas gives an example of a semicovering map which is not covering map
 (see \cite[Example 3.8]{B12}).
\item[(29)] By (16) in Diagram 1.
\item[(30)] By (15) in Diagram 1. The equality holds if $X$ is $T_1$ and paracompact space (see \cite[Theorem 3.13]{BF14}).
\end{itemize}

In order to investigate further the above diagram, we raise some questions in the following which we are interested in finding answers to them.

\begin{itemize}
\item[(Q1)] Is there a $\pi_1^{\pi^{gc}(X,x_0)}$- covering map which is not a $\pi_1^{\widetilde{\pi}^{sp}(X,x_0)}$-covering map?
\item[(Q2)] Under which condition the equality of $\pi_1^{\pi^{gc}(X,x_0)}$- covering maps and generalized covering maps holds?  
Is there a $\pi_1^{\pi^{gc}(X,x_0)}$- covering map which is not generalized covering map?
\item[(Q3)] Is there a $\pi_1^{wh}$- covering map which is not a $\pi_1^{qtop}$- covering map?
\item[(Q4)] Under which condition the equality of $\pi_1^{wh}$- covering maps and generalized covering maps holds? Is there a $\pi_1^{wh}$- covering map 
which is not a generalized covering map?
\item[(Q5)] Is there a $\pi_1^{\widetilde{\pi}^{sp}(X,x_0)}$- covering map which is not a $\pi_1^{pSpan}$- covering map?
\item[(Q6)] Is there a $\pi_1^{qtop}$- covering map which is not a $\pi_1^{pSpan}$- covering map?
\item[(Q7)] Is there a $\pi_1^{pSpan}$- covering map which is not a $\pi_1^{\pi^{sp}(X,x_0)}$- covering map?
\item[(Q8)] Is there a $\pi_1^{\pi^{fib}(X,x_0)}$- covering map which is not a $\pi_1^{\pi^{sp}(X,x_0)}$- covering map?
\item[(Q9)] Is there a $\pi_1^{pSpan}$- covering map which is not a semicovering map?
\item[(Q10)] Under which condition the equality of $\pi_1^{\pi^{fib}(X,x_0)}$- covering maps and fibrations holds? Is there a 
$\pi_1^{\pi^{fib}(X,x_0)}$-covering map which is not a fibration?
\item[(Q11)] Is there a $\pi_1^{Span}$- covering map which is not a $\pi_1^{tSpan}$-covering map?
\item[(Q12)] Under which condition the equality of generalized covering maps and fibrations with upl holds?  
Is there a generalized covering map which is not a fibration with upl?
\item[(Q13)] Is there a condition for the equality of semicovering maps and fibrations with upl over $X$?  Is there a semicovering map which is not a fibration with upl?
\item[(Q14)] Is there a $\pi_1^{tSpan}$- covering map which is not a $\pi_1^{sh}$- covering map?
\item[(Q15)] Is there a $\pi_1^{Span}$- covering map which is not a covering map?
\item[(Q16)] Under which condition the equality of $\pi_1^{\Pi^{sp}(X,x_0)}$- covering maps and $\pi_1^{tSpan}$- covering maps holds?  
Is there a $\pi_1^{\Pi^{sp}(X,x_0)}$- covering map which is not $\pi_1^{tSpan}$-covering map?
\item[(Q17)] Is there a $\pi_1^{\pi^{sp}(X,x_0)}$- covering map which is not $\pi_1^{\Pi^{sp}(X,x_0)}$-covering map?
\end{itemize}


 \end{document}